\documentclass{amsart}
\usepackage{amsmath,amsfonts,amssymb}%
\usepackage[all]{xy}
\usepackage{setspace}
\usepackage{comment}
\usepackage[dvipsnames]{xcolor}
\usepackage{url}

\usepackage{tikz}

\usepackage{amsmath, amsthm, amsfonts, amssymb, stmaryrd}%
\usepackage{mathrsfs}

\usepackage{enumerate}

\usepackage{xparse, expl3}

\usepackage{url}
\usepackage{hyperref}
\hypersetup{
	colorlinks = true,
	citecolor=Blue,
	urlcolor=Blue,
	linkcolor=Green
}
\usepackage[capitalize]{cleveref}  
\crefname{lemma}{Lemma}{Lemmas}
\crefname{corollary}{Corollary}{Corollaries}
\crefname{theorem}{Theorem}{Theorems}
\crefname{equation}{Equation}{Equations}
\crefname{example}{Example}{Examples}
\crefname{section}{Section}{Sections}
\crefname{subsection}{Section}{Sections}

\newcommand\rest[1][]{\EM{\upharpoonright_{#1}}}

\newcommand{\defn}[1]{{\textbf{#1}}}

\def\ZFC{\text{ZFC}}

\DeclareMathOperator{\dom}{dom}
\DeclareMathOperator{\range}{ran}

\DeclareMathOperator{\Rel}{Rel}
\DeclareMathOperator{\arity}{ar}

\DeclareDocumentCommand{\SeqSize}{d[] d()}
{
\EM{#1^{(#2)}}
}

\DeclareMathOperator{\SubStrop}{Sub}
\DeclareDocumentCommand{\SubStr}{d[] d()}
{
\IfNoValueTF{#1}
	{
        \EM{\SubStrop(#2)}
	}
	{
        \EM{\SubStrop_{#1}(#2)}
	}
}

\DeclareMathOperator{\SFop}{SF}
\DeclareDocumentCommand{\SF}{d[]}
{
\IfNoValueTF{#1}
	{
	\EM{\SFop}
	}
	{
	\EM{\SFop_{#1}}
	}
}

\DeclareMathOperator{\distmo}{\EM{d}}
\DeclareDocumentCommand{\dist}{d[]}
{
\IfNoValueTF{#1}
	{
	\EM{\distmo}
	}
	{
	\EM{\distmo_{#1}}
	}
}

\DeclareDocumentCommand{\paths}{d[]}
{
\IfNoValueTF{#1}
	{
	\EM{[\cdot]}
	}
	{
	\EM{[#1]}
	}
}

\DeclareDocumentCommand{\dual}{d()}
{
\IfNoValueTF{#1}
	{
	\EM{\hat{\ }}
	}
	{
	\EM{\hat{#1}}
	}
}

\DeclareDocumentCommand{\SizeAS}{d()}
{
\IfNoValueTF{#1}
	{
	\EM{|\cdot|}
	}
	{
	\EM{|#1|}
	}
}

\DeclareDocumentCommand{\compcK}{d[]}
{
\IfNoValueTF{#1}
	{
	\EM{\mathbb{K}}
	}
	{
	\EM{\mathbb{K}[#1]}
	}
}

\DeclareDocumentCommand{\AgeK}{d[]}
{
\IfNoValueTF{#1}
	{
	\EM{\mathbf{K}}
	}
	{
	\EM{\mathbf{K}[#1]}
	}
}

\DeclareDocumentCommand{\Rel}{d[]}
{
\IfNoValueTF{#1}
	{
	\EM{\mathcal{R}}
	}
	{
	\EM{\mathcal{R}_{#1}}
	}
}

\DeclareDocumentCommand{\Func}{d[]}
{
\IfNoValueTF{#1}
	{
	\EM{\mathcal{F}}
	}
	{
	\EM{\mathcal{F}_{#1}}
	}
}

\DeclareDocumentCommand{\ar}{d[]}
{
\IfNoValueTF{#1}
	{
	\EM{\arity}
	}
	{
	\EM{\arity_{#1}}
	}
}

\DeclareDocumentCommand{\Fn}{d<> d[] d()}
{
\IfNoValueTF{#3}
	{
	\EM{\textrm{Fn}(#1, #2)}
	}
	{
	\EM{\textrm{Fn}(#1, #2, #3)}
	}
}

\DeclareMathOperator{\Sym}{Sym}
\DeclareDocumentCommand{\Perm}{d()}
{
\EM{\Sym(#1)}
}

\newcommand{\cofinal}{\EM{\textrm{cf}}}

\makeatletter
\newcommand{\dotminus}{\mathbin{\text{\@dotminus}}}

\newcommand{\@dotminus}{%
  \ooalign{\hidewidth\raise1ex\hbox{.}\hidewidth\cr$\m@th-$\cr}%
}

\DeclareMathOperator{\Lang}{\mathscr{L}}

\DeclareMathOperator{\levop}{lev}
\DeclareDocumentCommand{\lev}{m}
{
\levop(#1)
}

\DeclareMathOperator{\join}{join}

\def\bp{{\EM{\mathbf{p}}}}

\def\bx{{\EM{\mathbf{x}}}}
\def\by{{\EM{\mathbf{y}}}}

\def\cM{{\EM{\mathcal{M}}}}
\def\cN{{\EM{\mathcal{N}}}}

\def\cS{{\EM{\mathcal{S}}}}
\def\cT{{\EM{\mathcal{T}}}}
\def\cW{{\EM{\mathcal{W}}}}
\def\cX{{\EM{\mathcal{X}}}}

\DeclareDocumentCommand{\DistMonLang}{}
{
\Lang_{\textrm{dm}}
}

\DeclareDocumentCommand{\GenMetLang}{}
{
\Lang_{\textrm{gmet}}
}

\DeclareMathOperator{\ccop}{cc}
\DeclareDocumentCommand{\cc}{m}
{
\mbf{\ccop}(#1)
}

\DeclareMathOperator{\cseqop}{seq}
\DeclareDocumentCommand{\cseq}{m}
{
\mbf{\cseqop}(#1)
}

\DeclareDocumentCommand{\mcomp}{m}
{
\ol{#1}
}

\DeclareDocumentCommand{\coinit}{m}
{
\mathrm{coinit}(#1)
}

\DeclareDocumentCommand{\cofinal}{m}
{
\mathrm{cf}(#1)
}

\DeclareDocumentCommand{\nzero}{m}
{
{#1}^{+}
}

\def\w{\EM{\omega}}

\def\Reals{\EM{\mbb{R}}}

\def\^{\EM{{}^{\And}}}

\def\Or{\EM{\vee}}
\def\And{\EM{\wedge}}

\def\<{\EM{\langle}}
\def\>{\EM{\rangle}}

\def\EM#1{\ensuremath{#1}}
\def\mbb#1{\EM{\mathbb{#1}}}
\def\mbf#1{\EM{\mathop{\pmb{#1}}}}

\def\ol#1{\EM{\overline{#1}}}

\def\st{\,:\,}
\def\:{\colon}

\providecommand{\dotdiv}{
  \mathbin{
    \vphantom{+}
    \text{
      \mathsurround=0pt 
      \ooalign{
        \noalign{\kern-.35ex}
        \hidewidth$\smash{\cdot}$\hidewidth\cr 
        \noalign{\kern.35ex}
        $-$\cr 
      }%
    }%
  }%
}

\DeclareDocumentCommand{\RightJustify}{m}{\hspace*{\fill}\mbox{#1}\penalty-9999\relax}

\DeclareDocumentCommand{\DeclareCounter}{m}%
{%
 	\ifcsname c@#1\endcsname%
	\else%
		\newcounter{#1}%
	\fi%
}

\DeclareDocumentCommand{\MyQED}{}{\qed}

\DeclareDocumentEnvironment{proof}{d() D[]{\MyQED} d<>}
{
\noindent\IfNoValueTF{#1}
{\emph{Proof.\!\!}}
{\emph{Proof\ #1.\ }}
}
{
\IfValueTF{#3}
	{%
	\ExplSyntaxOff
	\RightJustify{\EM{#2_{#3}}}
	\vspace{1ex}
	}
	{
	\RightJustify{\EM{#2}}
	\vspace{1ex}
	}
}

\DeclareCounter{ProofLabelcOUntEr}
\DeclareDocumentCommand{\ProofLabel}{}{%
\addtocounter{ProofLabelcOUntEr}{1}
\label{cUrrEntProoflAbEl\arabic{ProofLabelcOUntEr}}
}

\DeclareDocumentCommand{\ProofRef}{D<>{1}}
{%
\ref{cUrrEntProoflAbEl\arabic{ProofcOUntEr#1}}
}
\DeclareDocumentCommand{\ProofCref}{D<>{1}}
{%
\cref{cUrrEntProoflAbEl\arabic{ProofcOUntEr#1}}
}

\DeclareDocumentEnvironment{Proof}{d() D[]{\MyQED} D<>{1}}
{
\DeclareCounter{ProofcOUntEr#3}%
\setcounter{ProofcOUntEr#3}{\value{ProofLabelcOUntEr}}%
\begin{proof}(#1)[#2]<\ProofRef<#3>>%
}
{
\end{proof}
}

\ExplSyntaxOn
\def\TheoremDepth{section}

\DeclareDocumentCommand{\DeclareTheorem}{m o m o}{%
\IfNoValueTF{#4}
	{%
	\IfNoValueTF{#2}
		{%
		\newtheorem{#1vArIAblE}{#3}
		}
		{%
		\newtheorem{#1vArIAblE}[#2vArIAblE]{#3}
		}
	}
	{%
	\newtheorem{#1vArIAblE}{#3}[#4]%
	}
\newtheorem*{#1vArIAblE*}{#3}

\DeclareDocumentEnvironment{#1}{o o}

	{
	\IfValueT{##2}%
		{
		\begin{spacing}{##2}
		}
	\IfValueTF{##1}
		{
		\begin{#1vArIAblE}[##1]
		}
		{
		\begin{#1vArIAblE}
		}
	\ProofLabel
	}
	{
	\IfValueT{##2}%
		{
		\end{spacing}{##2}
		}
	\end{#1vArIAblE}
	}

\DeclareDocumentEnvironment{#1*}{o o}

	{
	\IfValueT{##2}%
		{
		\begin{spacing}{##2}
		}
	\IfValueTF{##1}
		{
		\begin{#1vArIAblE*}[##1]
		}
		{
		\begin{#1vArIAblE*}
		}
	}
	{
	\IfValueT{##2}%
		{
		\end{spacing}{##2}
		}
	\end{#1vArIAblE*}
	}
}
\ExplSyntaxOff

\theoremstyle{plain}
\DeclareTheorem{theorem}{Theorem}[\TheoremDepth]
\DeclareTheorem{proposition}[theorem]{Proposition}
\DeclareTheorem{lemma}[theorem]{Lemma}
\DeclareTheorem{corollary}[theorem]{Corollary}
\DeclareTheorem{claim}[theorem]{Claim}

\theoremstyle{definition}
\DeclareTheorem{definition}[theorem]{Definition}
\DeclareTheorem{axiom}[theorem]{Axiom}
\DeclareTheorem{convention}[theorem]{Convention}

\theoremstyle{remark}
\DeclareTheorem{remark}[theorem]{Remark}
\DeclareTheorem{example}[theorem]{Example}
\DeclareTheorem{question}[theorem]{Question}
\DeclareTheorem{conjecture}[theorem]{Conjecture}

\begin{document}

\title{Absoluteness of Fixed Points}

\begin{abstract}
We characterize those complete commutative positive linear ordered monoids $\cW$ such that whenever $f$ is a map from a Cauchy complete $\cW$-metric space to itself, the existence of a fixed point of $f$ is independent of the background model of set theory. 
\end{abstract}

\author{Nathanael Ackerman}
\address{Harvard University,
Cambridge, MA 02138, USA}
\email{nate@aleph0.net}

\author{Mostafa Mirabi}
\address{The Taft School, Watertown, CT 06795, USA}
\email{mmirabi@wesleyan.edu}



\maketitle

\section{Introduction}

Suppose $\cW$ is a complete commutative positive linear ordered monoid,  $\cM$ is a Cauchy complete $\cW$-space and $f\:\cM \to \cM$ is a non-expanding map. When trying to understand the structure of the $\cW$-dynamical system $(\cM, f)$ an important first step is to understand the structure of the fixed points of $f$. However, even the existence of a fixed point of the dynamical system can depend on the background model of set theory,  i.e. there are transitive models of set theory $V_0 \subseteq V_1$ and $\cW$-dynamical systems $(\cM, f) \in V_0$ where $(\cM, f)$ has no fixed point in $V_0$ but does in $V_1$. 

In this paper we then answer the question ``For what $\cW$ is it the case that the existence of a fixed point of a $\cW$-dynamical system is absolute between models of set theory?''.

In order to make sense of this question in \cref{Section: Relativization} we recall the notion of relativization of a higher order structure. The notion of relativization (introduced in \cite{RelGT}) gives a meaning to when a structure which satisfies a higher order property in a larger model of set theory is \emph{the same structure} as one that satisfies the higher order property in a smaller model of set theory. 

The quintessential example of relativization is that of the complete metric space $(\Reals, |\cdot - \cdot|)$. Suppose we have two models of set theory $V_0\subseteq V_1$. In general one often thinks of $(\Reals, |\cdot - \cdot|)^{V_0}$ and $(\Reals, |\cdot - \cdot|)^{V_1}$ as \emph{the same structure}, i.e. the real numbers, just in different models of set theory. The notion of relativization makes precise this intuition. 

With the notion of relativization in hand, in \cref{Section: Distance Monoids}
we discuss \emph{distance monoids}, i.e. the collection of distances in a generalize metric space. In \cref{Section: W-metric spaces} we recall the basic machinery needed to study Cauchy complete $\cW$-metric spaces where $\cW$ is a distance monoid. While the material in this section is not new, we include it for completeness and because most presentations we have found either are too restrictive on the distance set, or are significantly more general than we need (and hence more complicated), often dealing with categories enriched in a quantaloid. 

In \cref{Section:Absoluteness} we prove our upwards absoluteness results. I.e. that fixed point are always upwards absolute and that if the coinitiality of $\cW$ is $\w$ or $\cW$ is not \emph{continuous at $0$} then the existence of a fixed point in a $\cW$-dynamical system is absolute between models of set theory.

In contrast, in \cref{Section:Non-Absoluteness} we show that if the coinitiality of $\cW$ is uncountable then there is a Cauchy complete $\cW$-dynamical system which has no fixed points, but which does have fixed points in some larger model of set theory. 

\subsection{Preliminaries}

In this paper we will work in a background model of Zermelo-Frankel set theory with the axiom of choice (\ZFC). We will use $V$ and its variants to denote transitive subclasses of this background model of set theory such that $(V, \in) \models \ZFC$. We will refer to such a $V$ as a \defn{model of set theory}. If $\gamma$ is an ordinal we let $\cofinal{\gamma}$ denote the cofinality of $\gamma$. If $f\:A \to B$ is a function, we let $\range(f) = f``[A] \subseteq B$. Languages will be finitary first order containing a collection of sorts, constant symbols, function symbols and relation symbols. We will use letters in the calligraphy font to denote structures (in some language), i.e. $\cM$, $\cN$, etc. and the corresponding Roman letters to represent their underlying set, i.e. $M$, $N$, etc.

\section{Relativization}
\label{Section: Relativization}

In this section we give the definition of a relativization of a $\Lang$-structure $\cM$, with respect to a property $P$, to a larger model of set theory. Intuitively, if $V_0 \subseteq V_1$ are models of set theory and $\cM \in V_0$ is an $\Lang$-structure satisfying $P$ in $V_0$ then a structure $\cN \in V_1$ is the relativization of $\cM$ to $V_1$ with respect to $P$ if $\cN$ is the \emph{smallest} $\Lang$-structure in $V_1$ which both contains $\cM$ and satisfies $P$ (in $V_1$). 


For more details on the notion of relativization see \cite{RelGT} or \cite{EncodeMetric}. 

\begin{definition}
\label{Definition:Relativization}
Suppose the following hold.
\begin{itemize}
\item $V_0 \subseteq V_1$ are models of set theory.

\item $\Lang, \cM, A \in V_0$ where $\Lang$ is a language, $\cM$ is a $\Lang$-structure and $A$ is a set.

\item $P(x, y)$ is some formula of set theory such that $V_0 \models P(\cM, A)$. 
\end{itemize}
We say a $\Lang$-structure $\cN \in V_1$ is the \defn{relativization} of $\cM$ to $V_1$ with respect to $P(x, A)$ if the following hold. 
\begin{itemize}
\item[(a)] $V_1 \models P(\cN, A)$.

\item[(b)] $M \subseteq N$ and the identity map is a homomorphism, i.e. $\cM$ is a (not necessarily induced) substructure of $\cN$. 

\item[(c)] If $i\:\cN \to \cN$ is an embedding which is the identity on $M$ then $i$ is an automorphism, 

\item[(d)] Whenever $\cN^* \in V_1$ is a $\Lang$-structure such that 
\begin{itemize}
\item $V_1 \models P(\cN^*, A)$, 

\item $M \subseteq N^*$ and the identity is a homomorphism, 
\end{itemize}
then the there is a homomorphism $i\:\cN \to \cN^*$ which is the identity on $\cM$. 
\end{itemize}
\end{definition}
\cref{Definition:Relativization} (a) and (b) ensure that $\cN$ satisfies $P(x, A)$ and contains $M$. \cref{Definition:Relativization} (d) ensures that $\cN$ is the minimal such model, and \cref{Definition:Relativization} ensures for any two relativizations $\cN_0$, $\cN_1$ there is an isomorphism from $\cN_0$ to $\cN_1$ which is the identity on $M$.

\section{Distance Monoids}
\label{Section: Distance Monoids}

In this section we review the basic notions surrounding the space of distances for a generalized metric space. 

\begin{definition}
\label{Definition: Linear Ordered Monoid}
A \defn{commutative positively linear ordered monoid} (or \defn{distance monoid}) is a triple $\<W, +, 0, \leq\>$ such that 
\begin{itemize}
\item $\<W, +, 0\>$ is a monoid, 

\item $\<W, \leq\>$ is a linear lattice minimal element $0$, 

\item $(\forall a, b, c \in W)\, a \leq b \rightarrow c + a \leq c + b \And a + c \leq b + c$. 

\item For any $A \subseteq W$ and $b \in W$, $b + \bigwedge A = \bigwedge_{a \in A} a + b$. 
\end{itemize}
Let $\nzero{W} = W \setminus \{0\}$.  For $n \in \w$ and $a \in W$ we use the following shorthand. 
\[
n \cdot a = \underbrace{a + \dots + a}_{n\textrm{ times}}
\]

We say $\cW$ is \defn{continuous at $0$} if $0 = \bigwedge_{a, b \in  \nzero{W}} a + b$. 
We say $\cW$ is \defn{complete} if $\<W,\leq\>$ is a complete linear order. 
We say $\cW$ is \defn{continuous} if it is complete and continuous at $0$. 
\end{definition}

We consider distance monoids as structure in the language $\DistMonLang= \{S_{\mathrm{dm}}, +, 0, \leq\}$ where $S_{\mathrm{dm}}$ is the unique sort, $+$ is a function symbol of arity $S_{\mathrm{dm}} \times S_{\mathrm{dm}} \to S_{\mathrm{dm}}$, where $0$ is a constant symbol of arity $S_{\mathrm{dm}}$, and where $\leq$ is a relation symbol of arity $S_{\mathrm{dm}} \times S_{\mathrm{dm}}$. 

One of our main objects of study in this paper will be generalized metric spaces whose distances take values in a distance monoid. It will also be important to understand what the Cauchy completion of such a generalized metric space is. However, if $\cM$ is such a generalized metric space with distances in $\cW$ and $\cW$ is not continuous at $0$ then there will not be any Cauchy sequence over $\cM$.

\begin{definition}
We define an \defn{initial sequence} in $\cW$ to be an order preserving map $s\:(\alpha^{op}, \leq) \to (\nzero{W}, \leq)$ such that $(\forall a \in \nzero{W})(\exists \beta \in \alpha)s(\beta) \leq a$.

We define the \defn{coinitiallity} of $\cW$, denoted $\coinit{\cW}$, to be the smallest ordinal which is the domain of an initial sequence.

We call an initial sequence $\alpha$ \defn{nice} if $(\forall i \in \dom(\alpha))\, 4 \cdot \alpha(i+1) \leq \alpha(i)$.  
\end{definition}

\begin{lemma}
The coinitiality of a continuous distance monoid is an infinite regular cardinal. 
\end{lemma}
\begin{proof}
Let $\cW$ be a distance monoid. It is immediate that $\coinit{\cW}$ is a regular cardinal. Further if $\coinit{\cW}$ is finite it must be $1$. But if $\coinit{\cW} = 1$ then $\bigwedge \nzero{\cW}$ exists, as $\cW$ is complete, and $\bigwedge \nzero{\cW} > 0$ hence $\cW$ is not continuous at $0$. 
\end{proof}

The following shows that we can always find nice initial sequences for continuous distance monoid. 
\begin{proposition}
\label{Nice initial sequences exist}
Suppose $\cW$ is a continuous distance monoid. 
\begin{itemize}
\item[(a)] If $\coinit{\cW} = \w$ then for all $n \in \w$ there is an initial sequence $\alpha$ such that 
\[
(\forall k \in \w)\, n \cdot \alpha(k+1) \leq \alpha(k).
\]

\item[(b)] If $\coinit{\cW} > \w$ then there is an initial sequence $\alpha$ such that 
\[
(\forall n \in \w)(\forall k \in \coinit{\cW})\, n \cdot \alpha(k+1) \leq \alpha(k).
\]

\end{itemize}

\end{proposition}
\begin{proof}
As $\cW$ is continuous at $0$, for each $a \in \nzero{W}$ we can find $h(a) \in \nzero{W}$ such that $2\cdot h(a) \leq a$. In particular we have $2^n \cdot h^n(a) \leq a$. Let $h^\w(a) = \bigwedge h^n(a)$. Suppose $\beta\:\coinit{\cW}^{op} \to \nzero{W}$ is an initial sequence. 

If $\coinit{\cW} = \w$ then define an initial sequence $\alpha$ as follows. First $\alpha(0) = \beta(0)$. If $\alpha(k)$ is defined let $\alpha(k+1)$ be any $\beta(j)$ such that $\beta(j) \leq h^n(\alpha(k))$ and $j \geq k$. It is then immediate that for all $k \in \w$, we have $n \cdot \alpha(k+1) \leq 2^n \cdot \alpha(k+1) \leq \alpha(k)$ and so (a) holds. 

Now suppose $\coinit{\cW} > \w$. We define an initial sequence $\alpha$ as follows. Let $\alpha(0) = \beta(0)$. Let $\alpha(\w \cdot \delta) = \bigwedge_{i \in \w \cdot \delta} \alpha(i)$. Suppose $\alpha(\gamma)$ has been defined. As $\coinit{\cW} > \w$ we have $h^\w(\alpha(\gamma)) > 0$. Let $\alpha(\gamma+1) = \beta(j)$ for any $\beta(j) \leq h^\w(\alpha(\gamma))$ and $j \geq \gamma$. It is then immediate that $\alpha$ satisfies the conditions of (b). 
\end{proof}

\begin{definition}
If $\cW = \<W, +, 0, \leq\>$ is a distance monoid and $\cW^* = \<W^*, +^*, 0, \leq^*\>$ is a distance monoid such that $\<W^*,\leq^*\>$ is the Dedekind–MacNeille completion of $\<W, \leq\>$ and $+^* \cap (W \times W) = +$ then we call $\cW^*$ a \defn{completion} of $\cW$.  
\end{definition}

\begin{proposition}
\label{Completions of distance monoid exists}
If $\cW$ is a distance monoid then up to isomoprhism it has a unique completion $\mcomp{\cW}$. If $\cX$ is a complete distance monoid and $i\:\cW \to \cX$ is an $\DistMonLang$-embedding then there is a unique embedding $j\:\mcomp{\cW} \to \cX$ such that $j \rest[W] = i$. 
\end{proposition}
\begin{proof}
For $a \in \mcomp{\cW}$ we have $a = \bigwedge_{b \geq a, b \in W}$. Let $a +^* b = \bigwedge_{a' \geq a, b' \geq b, a, b \in W} a' + b'$. It is then straightforward to check that $\<W^*, +^*, 0, \leq^*\>$ is a distance monoid. Further, if $\<W^*, +^\circ, 0, \leq^*\>$ is any completion of $\cW$ and $a, b \in W^*$ then 
\begin{align*}
a +^\circ b & = a +^\circ \bigwedge_{b' \geq b, b \in W} = \bigwedge_{b' \geq b, b \in W} a +^\circ b' = \bigwedge_{b' \geq b, b \in W} (b' + \bigwedge_{a' \geq a, a \in W} a' +^\circ b') \\
& = \bigwedge_{b' \geq b, b \in W}\bigwedge_{a' \geq a, a \in W} a' +^\circ b' = a +^* b.
\end{align*}
Hence $+^\circ = +^*$ and the completion is unique. 

Using the universal properties of the Dedekind–MacNeille completion it is then straightforward to show that whenever $i\:\cW \to \cX$ is an embedding there is a unique extension to an embedding $j\:\mcomp{\cW} \to \cX$. 
\end{proof}

The following is immediate. 
\begin{lemma}
Suppose $\cW$ is a distance monoid. 
\begin{itemize}
\item $\coinit{\cW} = \coinit{\mcomp{\cW}}$. 

\item If $\cW$ is continuous at $0$ then $\mcomp{\cW}$ is continuous at $0$.

\end{itemize}
\end{lemma}

Note the following is an important and immediate consequence of \cref{Completions of distance monoid exists}.

\begin{proposition}
\label{Relativization of continuous distance monoids}
Suppose
\begin{itemize}
\item $V_0 \subseteq V_1$ are models of set theroy. 

\item $\cW \in V_0$ and $V_0 \models \cW\text{ is a complete distance monoid}$. 

\item $\cW^* = \mcomp{\cW}^{V_1}$. 
\end{itemize}
Then $\cW^*$ is the relativization of $\cW$ to $V_1$ with respect to the property of being a complete distance monoid. 
\end{proposition}

\section{$\cW$-Metric Spaces}
\label{Section: W-metric spaces}

In this section we review the basic concepts surrounding generalized metric spaces which take values in a distance monoid. 

\begin{definition}
Suppose $\cW$ is a distance monoid. A \defn{$\cW$-metric space} is a pair $\cM = (M, \dist[\cM])$ where $M$ is a set and $\dist[\cM] \: M \times M \rightarrow W$ satisfies:
\begin{itemize}
\item[] $(\forall x, y\in M)\, [\dist[\cM](x, y) \Or \dist[\cM](y, x) = 0] \leftrightarrow x = y$. 

\item[] (Transitivity) $(\forall x, y, z \in M)\, \dist[\cM](x, z) \leq \dist[\cM](x, y) + \dist[\cM](y, z)$.

\end{itemize}

By a \defn{generalized metric space} we mean a pair $(\cM, \cW)$ where $\cW$ is a distance monoid and $\cM$ is a $\cW$-metric space.

We consider a generalized metric space as a $\GenMetLang$-structure where $\GenMetLang = \DistMonLang \cup \{S_{\mathrm{met}}, \dist\}$ where $S_{\mathrm{met}}$ is a sort and $\dist$ is a function symbol of arity $S_{\mathrm{met}} \times S_{\mathrm{met}} \to S_{\mathrm{dm}}$. 
\end{definition}

\begin{definition}
Suppose $\cM$ is a $\cW$-metric space and $f:M \rightarrow M$. We say $f$ is \defn{non-expanding} if
\[
(\forall x,y\in M)\, \dist[\cM](f(x), f(y)) \leq \dist[\cM](x, y).
\]
\end{definition}
%

\subsection{Cauchy Completion}

In this section we give basic properties of Cauchy complete $\cW$-metric spaces. 

\begin{definition}
Suppose $\cW$ is a distance monoid and $\cM$ is a $\cW$-metric space. We say $p$ is a \defn{Cauchy sequence} on $\cM$ if
\begin{itemize}
\item $p\: A \to M$ where $A \subseteq \nzero{W}$ and $\bigwedge_{a, b \in A} a + b = 0$. 

\item $(\forall a, b \in A)\, \dist[\cM](p(a), p(b)) \leq a + b$

\end{itemize}
We call $A$ the \defn{domain} of $\bp$. 

We say that $p$ \defn{converges} to $x\in M$ if 
\[
\bigvee_{a \in \nzero{W}} \bigwedge_{\substack{b \leq a\\ b \in \dom(p)}} \dist[\cM](x, p(b)) = 0\text{ and } \bigwedge_{\substack{b \leq a\\ b \in \dom(p)}} \dist[\cM](x, p(b)) \dist[\cM](p(b), x) = 0
\]

We say a Cauchy sequence is \defn{maximal} if $A = \nzero{W}$. 

We say a generalized metric space $(\cM, \cW)$ is \defn{Cauchy complete} if $\cW$ is complete and every Cauchy sequence converges to an element in $M$.
\end{definition}

The following lemma is one of the main reasons we focus on continuous distance monoids. 
\begin{lemma}
\label{Non-continuous distance monoids give rise to Cauchy complete spaces}
Suppose $\cW$ is a complete distance monoid which is not continuous at $0$ and $\cM$ is a $\cW$-metric space. Then $(\cM, \cW)$ is Cauchy complete. 
\end{lemma}
\begin{proof}
Because $\cW$ is not continuous at $0$ there are no sets $A \subseteq \nzero{W}$ with $\bigwedge_{a, b \in A} a + b = 0$. Therefore there are no Cauchy sequences on $\cM$.
\end{proof}

\begin{definition} 
Suppose $\cW$ is a continuous distance monoid and $\cM$ is a $\cW$-metric space. Let $\cseq{\cM}$ be the collection of Cauchy sequences over $\cM$. 

For $p, q \in \cseq{\cM}$ define
\[
\dist[\cseq{\cM}](p, q) = \bigvee_{c \in \nzero{W}} \bigwedge_{\substack{a\in \dom(p)\\ b \in \dom(q) \\ a, b \leq c}}\dist[\cM](p(a), q(b)).
\]

We say $p \equiv_{\cM} q$ if $\dist[\cseq{\cM}](p, q) = 0$ and $\dist[\cseq{\cM}](q, p) = 0$. 
\end{definition}

\begin{lemma}
\label{Transitivity of distance on Cauchy sequences}
Suppose $\cW$ is a continuous distance monoid, $\cM$ is a $\cW$-metric space and $p, q, r \in \cseq{\cM}$. Then 
\[
\dist[\cseq{\cM}](p, r)  \leq \dist[\cseq{\cM}](p, q) + \dist[\cseq{\cM}](q, r).
\]
\end{lemma}
\begin{proof}
For all $c \in \nzero{W}$ we have 
\begin{align*}
 \dist[\cseq{\cM}](p, q) & + \dist[\cseq{\cM}](q, r) + 2 \cdot c  \\
& \geq\!\!\!\! \bigwedge_{\substack{a\in \dom(p)\\ e \in \dom(q) \\ a, e \leq c}}\!\!\!\! \dist[\cseq{\cM}](p(a), q(e)) + \!\!\!\! \bigwedge_{\substack{a\in \dom(p)\\ e \in \dom(q) \\ a, e \leq c}}\!\!\!\! \dist[\cseq{\cM}](q(e), r(b)) + 2 \cdot c\\
\geq &  \!\!\!\! \bigwedge_{\substack{a\in \dom(p)\\ e_0, e_1 \in \dom(q) \\ b \in \dom(r) \\a, b, e_0, e_1\leq c}}\!\!\!\!\dist[\cseq{\cM}](p(a), q(e_0)) + e_0 + e_1 + \dist[\cseq{\cM}](q(e_1), r(b))\\
\geq & \bigwedge_{\substack{a\in \dom(p)\\ e_0, e_1 \in \dom(q) \\ b \in \dom(r) \\a, b, e_0, e_1\leq c}} \!\!\!\!\dist[\cseq{\cM}](p(a), q(e_0)) + \dist[\cseq{\cM}](q(e_0), q(e_1)) + \dist[\cseq{\cM}](q(e_1), r(b)) \\
\geq & \!\!\!\!\bigwedge_{\substack{a\in \dom(p)\\ b \in \dom(r) \\a, b\leq c}}\!\!\!\!\dist[\cseq{\cM}](p(a), r(b))
\end{align*}

But as $c\in \nzero{W}$ was arbitrary this implies 
\[
\dist[\cseq{\cM}](p, r)  \leq \dist[\cseq{\cM}](p, q) + \dist[\cseq{\cM}](q, r).
\]
as desired.
\end{proof}

The following is standard and we include a brief proof for completeness.
\begin{lemma}
\label{Properties of equivalence of Cauchy sequences}
Suppose $\cW$ is a continuous distance monoid and $p, q, r$ are Cauchy sequences. 
\begin{itemize}
\item[(a)] $p \equiv_{\cM} p$. 

\item[(b)] $p \equiv_{\cM} q \rightarrow q \equiv_{\cM} p$. 

\item[(c)] $p \equiv_{\cM} q \And q \equiv_{\cM} r \rightarrow p \equiv_{\cM} r$. 

\item[(d)] $p_0 \equiv_{\cM} q_0 \And p_1 \equiv_{\cM} q_1 \rightarrow \dist[\cseq{\cM}](p_0, q_0) = \dist[\cseq{\cM}](p_1,q_1)$. 

\item[(e)] $p$ converges to $x$ if and only if $p \equiv_{\cM} \widehat{x}$ where $\widehat{x}\:\nzero{W} \to W$ is the constant function with value $x$. 

\item[(f)] There is a maximal Cauchy sequence $p^*$ with $p^* \equiv_{\cM} p$. 

\item[(g)] If $p$ is a Cauchy sequence and $\alpha$ is an initial sequence with $\range(\alpha) \subseteq \dom(p)$ then $p\rest[\range(\alpha)]$ is a Cauchy sequence and $p \equiv_{\cM} p \rest[\range(\alpha)]$. 
\end{itemize}
\end{lemma}
\begin{proof}
Note (a) is immediate from the definition of a Cauchy sequence,  (b) is immediate from the definition of $\equiv_\cM$, (c) and (d) follows from \cref{Transitivity of distance on Cauchy sequences} and (e) is immediate from the definition of a Cauchy sequence converging to an element.  (g) is immediate from the definition of an initial sequence. 

To see (f) suppose $p$ is a non-maximal Cauchy sequence. Let $p^*\:\nzero{W} \to M$ be any map such that $(\forall a \in \nzero{W})(\exists b \in \dom(p))\, b \leq a \text{ and }p^*(a) = p(b)$. It is then straightforward to check $p^*$ is a Cauchy sequence with $p \equiv_{\cM} p^*$. 
\end{proof}

In particular, $\cM$ is Cauchy complete if and only if every maximal Cauchy sequence converges to an element.  The following is immediate. 
\begin{lemma}
\label{Embedding of generalized metric space into completion}
Suppose $\cW$ is a continuous distance monoid and $\cM$ is a $\cW$-metric space. For $x \in M$ let $\widehat{x}\:\nzero{W} \to M$ be the constant function with value $x$. 
\begin{itemize}
\item For all $x \in M$, $\widehat{x} \in \cseq{\cM}$. 

\item For all $x, y \in M$, $\dist[\cM](x, y) = \dist[\cseq{\cM}](\widehat{x}, \widehat{y})$. 

\end{itemize}
\end{lemma}
By \cref{Properties of equivalence of Cauchy sequences} there is a $\cW$-space $(\cseq{\cM}/\equiv_{\cM}, \dist[\cseq{\cM}]^*)$ where, for all $x, y \in \cseq{\cM}/\equiv_{\cM}$, $\dist[\cseq{\cM}]^*(x, y) = \dist[\cseq{\cM}](a, b)$ for any $a \in x$ and $b \in y$. 
\cref{Embedding of generalized metric space into completion} says the map $x \mapsto \widehat{x}$ is an embedding of $(M, \dist[\cM])$ into $(\cseq{\cM}/\equiv_{\cM}, \dist[\cseq{\cM}]^*)$ which preserves distance.  

\begin{definition}
\label{Cauchy completion with continuous distance monoids}
Suppose $\cW$ is a continuous distance monoid and $\cM$ is a $\cW$-metric space. Let 
\[
\cc{M} = M \cup \{x \in \cseq{\cM}/\equiv_{\cM} \st (\forall a \in M)\, \widehat{a} \not \in x\}.
\]

Let $\dist[\cc{\cM}]\: \cc{M} \times \cc{M} \to W$ be the map where the following hold.  
\begin{itemize}
\item For $x,y \in M$ let $\dist[\cc{\cM}](x, y) = \dist[\cM](x, y)$. 

\item For $x \in M$ and $y \in \cc{M} \setminus M$ let $\dist[\cc{\cM}](x, y) = \dist[\cseq{\cM}](\widehat{x}, y)$ and $\dist[\cc{\cM}](x, y) = \dist[\cseq{\cM}](y, \widehat{x})$. 

\item For $x, y \in \cc{M} \setminus M$ let $\dist[\cc{\cM}](x, y) = \dist[\cseq{\cM}](x, y)$. 
\end{itemize}
We let $\cc{\cM} = (\cc{M}, \dist[\cc{\cM}])$ and call it the \defn{Cauchy completion} of $\cM$. We call $(\cM, \cW)$ the \defn{Cauchy completion} of $(\cM, \cW)$. 
\end{definition}

By \cref{Embedding of generalized metric space into completion} there is isomorphism $\iota\:\cc{\cM} \to (\cseq{\cM}/\equiv_{\cM}, \dist[\cseq{\cM}]^*)$ where $\iota(x) = \widehat{x}$ if $x \in M$ and $\iota(x) = x$ if $x \not \in M$. The reason why we define the Cauchy completion to be as in \cref{Cauchy completion with continuous distance monoids} instead of $(\cseq{\cM}/\equiv_{\cM}, \dist[\cseq{\cM}]^*)$ is it will be important from the perspective of  relativization that $M$ is contained in $\cc{M}$ as a subset.

The following is standard and we provide the proof for completeness. 
\begin{proposition}
\label{Cauchy completion is the completion}
Suppose $\cW$ is continuous distance monoid and $\cM$ is a $\cW$-metric space. 
\begin{itemize}
\item[(a)] $\cc{\cM}$ is a Cauchy complete $\cW$-metric space.

\item[(b)] If $i\:\cM \to \cX$ is a non-expanding map where $\cX$ is a Cauchy complete $\cW$-metric space then there is a unique map $j\:\cc{\cM} \to \cX$ such that $j \rest[M] = i$.

\item[(c)] There is an isomorphism from $\cc{\cc{\cM}}$ to $\cc{\cM}$ which is the identity on $M$. 
\end{itemize}
\end{proposition}
\begin{proof}
First note that (c) follows immediately from (a) and (b). 

We now show (a). Note $\cc{\cM}$ is a $\cW$-metric space by \cref{Transitivity of distance on Cauchy sequences}. Suppose $p$ is a maximal Cauchy sequence in $\cc{\cM}$. 
Let $\widetilde{p}\:\nzero{W} \to \cc{M}$ be the map where $\widetilde{p}(a) = p(a)$ if $p(a) \not \in M$ and $\widetilde{p}(a) = \widehat{p(a)}$ if $p(a) \in M$. Let $p^*\:\nzero{W} \to \cc{M}$ be the map where $p^*(a) = \widetilde{p}(a)(a)$ for all $a \in \nzero{W}$. 

\begin{claim}
$p^* \in \cc{\cM}$. 
\end{claim}
\begin{proof}
Suppose $a, b, c, e, t\in \nzero{W}$. Then 
\begin{align*}
\dist[\cM](p^*(a), p^*(b)) & = \dist[\cM](\widehat{p}(a)(a), \widetilde{p}(b)(b)) \\
&= \dist[\cM](\widetilde{p}(a)(a), \widetilde{p}(a)(c)) + \dist[\cM](\widetilde{p}(a)(c), \widetilde{p}(b)(e)) + \dist[\cM](\widetilde{p}(b)(e), \widetilde{p}(b)(b))\\
&\leq a + c + \dist[\cM](\widetilde{p}(a)(c), \widetilde{p}(b)(e))+ b + e .
\end{align*}
Therefore 
\begin{align*}
\dist[\cM](p^*(a), p^*(b)) & \leq \bigvee_{\substack{s \in \nzero{W} \\ s \leq t}}\bigwedge_{\substack{c, e \in \nzero{W}
\\ c, e \leq s}} \dist[\cM](\widetilde{p}(a)(c), \widetilde{p}(b)(e)) + a + c + b + e \\
&  \leq 4 \cdot t +  \bigvee_{\substack{s \in \nzero{W} \\ s \leq t}}\bigwedge_{\substack{c, e \in \nzero{W}
\\ c, e \leq s}}\dist[\cM](\widetilde{p}(a)(c), \widetilde{p}(b)(e))  \\
& = 4 \cdot t + \dist[\cc{\cM}](\widetilde{p}(a), \widetilde{p}(b)) \leq 4 \cdot t + a + b. 
\end{align*}
But because $t$ was arbitrary this implies $\dist[\cM](p^*(a), p^*(b)) \leq a + b$ and the claim holds. 
\end{proof}

\begin{claim}
$p$ converges to $p^*$. 
\end{claim}
\begin{proof}
Let $a , t\in \nzero{W}$. We then have the following. 
\begin{align*}
\dist[\cc{\cM}](p^*, p(a)) & = \dist[\cseq{\cM}](p^*, \widetilde{p}(a)) \\
& =\bigvee_{s \in \nzero{W}} \bigwedge_{\substack{b, c \in \nzero{W}\\ b, c \leq s}} \dist[\cM](p^*(b), \widetilde{p}(a)(c)) 
=\bigvee_{\substack{s \in \nzero{W}\\ s \leq t}} \bigwedge_{\substack{b, c \in \nzero{W}\\ b, c \leq s}} \dist[\cM](p^*(b), \widetilde{p}(a)(c))\\%
&= \bigvee_{\substack{s \in \nzero{W}\\ s \leq t}} \bigwedge_{\substack{b, c \in \nzero{W}\\ b, c \leq s}} \dist[\cM](\widetilde{p}(b)(b), \widetilde{p}(a)(c)) \\
& \leq \bigvee_{\substack{s \in \nzero{W}\\ s \leq t}} \bigwedge_{\substack{b, c, e \in \nzero{W}\\ b, c, e \leq s}} \dist[\cM](\widetilde{p}(b)(b), \widetilde{p}(b)(e)) + \dist[\cM](\widetilde{p}(b)(e), \widetilde{p}(a)(c)) \\
& = \bigvee_{\substack{s \in \nzero{W}\\ s \leq t}} \bigwedge_{\substack{b, c, e \in \nzero{W}\\ b, c, e \leq s}} b + e + \dist[\cM](\widetilde{p}(b)(e), \widetilde{p}(a)(c)) \\
& \leq \bigvee_{\substack{s \in \nzero{W}\\ s \leq t}} \bigwedge_{\substack{b, c, e \in \nzero{W}\\ b, c, e \leq s}} 2 \cdot t + \dist[\cM](\widetilde{p}(b)(e), \widetilde{p}(a)(c)) \\
& =  2\cdot t+ \bigvee_{\substack{s \in \nzero{W}\\ s \leq t}} \bigwedge_{\substack{b \in \nzero{W}\\ b \leq s}} \dist[\cc{\cM}](\widetilde{p}(b), \widetilde{p}(a)) \\
& \leq 2\cdot t +  \bigwedge_{\substack{b \in \nzero{W}\\ b \leq s}} a + b  = 2 \cdot t + a.
\end{align*}
As $t$ was arbitrary we have $\dist[\cc{\cM}](p^*, p(a)) \leq a$ and so $p$ converges to $p^*$. 
\end{proof}
As $p$ was arbitrary $\cc{\cM}$ is Cauchy complete and so (a) holds.

We now show (b). Suppose $\cX$ is Cauchy complete and $i\:\cM \to \cX$ is a non-expanding map. Suppose  $j_0, j_1\:\cc{\cM} \to \cX$ are non-expanding maps with $j_0\rest[M] = j_1\rest[M] = i$. Let $x \in \cc{M}\setminus M$ with $x^* \in x$ be maximal and $a \in \nzero{W}$. Then 
\begin{align*}
\dist[\cX](j_0(x), j_1(x))  & \leq \dist[\cX](j_0(x), j_0(x(a))) + \dist[\cX](j_0(x(a)), j_1(x(a))) + \dist[\cX](j_1(x(a)), j_1(x))\\
&\leq \dist[\cc{\cM}](x, x(a)) + 0 + \dist[\cc{\cM}](x(a), x).
\end{align*}
But as $x$ is a Cauchy sequence in $\cM$ this implies $\dist[\cX](j_0(x), j_1(x)) = 0$. An identical argument shows $\dist[\cX](j_1(x), j_0(x)) = 0$ and so $j_0(x) = j_1(x)$. In particular if there is a non-expanding $j\:\cc{\cM} \to \cX$ extending $i$ it must be unique.

Suppose $x \in \cc{\cM}$. If $x \in M$ let $J(x) = \widehat{i(x)}$. If $x \in \cc{M} \setminus M$ let $x^* \in x$ be maximal and define $J(x)\:\nzero{W} \to X$ to be such that $J(X)(a) = i(x^*(a))$ for all $a \in \nzero{W}$.

Suppose $x \in \cc{M} \setminus M$ and $x^* \in x$ is maximal. Then 
\[
\dist[\cX](J(x)(a), J(x)(b)) \leq \dist[\cM](x^*(a), x^*(b)) \leq a + b.
\]
Therefore $J(x)$ is a Cauchy sequence in $\cX$. Let $j(x)$ be the element $J(x)$ converges to, which must exist as $\cX$ is Cauchy complete. 

Next suppose $x,y \in \cc{\cM}$. Let $x^* \in x$ be maximal if $x \not \in M$ and $x^* = \widehat{x}$ if $x \in M$. Define $y^*$ similarly. Then 
\begin{align*}
\dist[\cX](j(x), j(y)) & = \bigvee_{s \in \nzero{W}} \bigwedge_{\substack{c, e \in \nzero{W} \\ c, e \leq s}} \dist[\cX](J(x)(c), J(y)(e)) \\
& \leq \bigvee_{s \in \nzero{W}} \bigwedge_{\substack{c, e \in \nzero{W} \\ c, e \leq s}} \dist[\cM](x^*(c), y^*(e)) = \dist[\cc{\cM}](x, y).
\end{align*}
Therefore (b) holds. 
\end{proof}

As $\cM$ is always Cauchy complete if $\cW$ is not continuous at $0$ the following definition will be useful. 

\begin{definition}
Suppose $\cW$ is a distance monoid which is not continuous at $0$ and $\cM$ is a $\cW$-metric space. We then let $\cc{\cM}$ be the $\mcomp{\cW}$-metric space with underlying set $M$ and such that $\dist[\cc{\cM}](x, y) = \dist[\cM](x, y)$ for all $x, y \in M$, i.e. $\cc{\cM}$ is $\cM$ considered as a $\mcomp{\cW}$-metric space. 
\end{definition}

\begin{definition}
Suppose $\cW$ is a distance monoid $\cM$ is a $\cW$-metric space. We call $(\cc{\cM}, \mcomp{\cW})$ the \defn{Cauchy completion} of $(\cM, \cW)$. When $\cW$ is complete we call $\cc{\cM}$ the \defn{Cauchy completion} of $\cM$. 
\end{definition}

Note the following is immediate from \cref{Relativization of continuous distance monoids} and \cref{Cauchy completion is the completion} in the case $\cW$ is a continuous distance monoid and \cref{Non-continuous distance monoids give rise to Cauchy complete spaces} in the case $\cW$ is a complete distance monoid which is not continuous at $0$. 
\begin{proposition}
\label{Relativization of complete generalized metric space}
Suppose the following hold.
\begin{itemize}
\item $V_0 \subseteq V_1$ are models of set theory.

\item $(\cM, \cW) \in V_0$.

\item $V_0 \models \cW$ is a complete distance monoid and $\cM$ is a Cauchy complete $\cW$-metric space. 

\item $(\cN, \cX) \in V_1$ is the Cauchy completion of $(\cM, \cW)$ in $V_1$. 
\end{itemize}
Then $(\cN, \cX)$ is the relativization of $(\cM, \cW)$ to $V_1$ with respect to the property of being a Cauchy complete generalized metric space. 
\end{proposition} 

\begin{definition}
We say a set $D \subseteq \cM$ is \defn{dense} if there is a isomorphism $\cc{D, \dist[\cM]}$ to $\cc{\cM}$ which is the identity on $D$.
\end{definition}

The following is immediate from \cref{Cauchy completion is the completion}. 
\begin{lemma}
\label{Dense preserved by taking Cauchy completion}
If $D$ is dense in $\cM$ then $D$ is dense in $\cc{\cM}$. 
\end{lemma}

Note that $V_0 \subseteq V_1$ are models of set theory and $\cM$ is any $\cW$-metric space in $V_0$ then $(\cc{\cc{\cM, \cW}^{V_0}})^{V_1} = \cc{\cM, \cW}^{V_1}$.. In particular being a dense subset of a Cauchy complete generalized metric space is absolute.

\begin{lemma}
\label{Non-expanding maps have unique extensions to Cauchy complete generalized metric spaces}
Suppose $f\:\cM\to \cN$ is a non-expanding map. 
\begin{itemize}
\item[(a)] There is a unique map $\cc{f}\:\cc{\cM} \to \cc{\cN}$ which agrees with $f$ on $M$. 

\item[(b)] If $p \in x \in \cc{M} \setminus M$ then either 
\begin{itemize}
\item $f \circ p \in f(x) \in \cc{N} \setminus N$,  or 

\item $f \circ p$ converges to $f(x)$. 
\end{itemize}
\end{itemize}
\end{lemma}
\begin{proof}
Condition (a) is immediate from \cref{Cauchy completion is the completion}. Condition (b) is immediate from the fact that if $p \in x \in \cc{M} \setminus M$ then $f \circ p$ is a Cauchy sequence in $N$. 
\end{proof}

The following will be one of the main objects of study in this paper. 
\begin{definition}
Suppose $\cW$ is a complete distance monoid. A \defn{$\cW$-dynamical system} is a pair $(\cM, f)$ where $\cM$ is a Cauchy complete $\cW$-metric space and $f\:\cM \to \cM$ is a non-expanding map. By a \defn{generalized dynamical system} we mean a triple $(\cM, f, \cW)$ where $\cW$ is a complete distance monoid and $(\cM, f)$ is a $\cW$-dynamical system. 
\end{definition}

The following is immediate from \cref{Relativization of complete generalized metric space} and \cref{Non-expanding maps have unique extensions to Cauchy complete generalized metric spaces}. 
\begin{proposition}
If $(\cM, f, \cW) \in V_0$ is generalized dynamical system then the generalized dynamical system $(\cc{\cM}, \cc{f}, \mcomp{\cW})^{V_1}$ is the relativization of $(\cM, f, \cW)$ to $V_1$ with respect to being a generalized dynamical system. 
\end{proposition}

\begin{definition}
We say a generalized dynamical system $(\cM, f, \cW)$ \defn{has a fixed point} if there is an $x \in M$ such that $f(x) = x$. 
\end{definition}

The main focus of this paper will be to study when having a fixed point is absolute for a generalized dynamical system.

\section{Absoluteness}
\label{Section:Absoluteness}

In this section we consider when the existence of a fixed point is absolute. For the rest of the paper let $V_0 \subseteq V_1$ be models of set theory, let $(\cM_0, f_0, \cW_0) \in V_0$ be a generalized dynamical system, and let $(\cM_1, f_1, \cW_1)\in V_1$ be its relativization to $V_1$ for being a generalized dynamical system. 

\begin{proposition}
If $(\cM_0, f_0)$ has a fixed point then so does $(\cM_1, f_1)$. 
\end{proposition}
\begin{proof}
This is because $\cM_0 \subseteq \cM_1$ and so if $x \in \cM_0$ with $f_0(x) = x$ we also have $x \in \cM_1$ with $f_1(x) = x$. 
\end{proof}

In other words a generalized dynamical system having a fixed point is upwards absolute.  

\begin{proposition}
\label{Upwards absoluteness when distance monoid is not continuous at 0}
Suppose $\cW_0$ is not continuous at $0$. Then $(\cM_0, f_0)$ has a fixed point if and only if $(\cM_1, f_1)$ has a fixed point. 
\end{proposition}
\begin{proof}
By \cref{Relativization of complete generalized metric space} we have $(\cM_0, f_0) = (\cM_1, f_1)$. The result then follows. 
\end{proof}

\begin{theorem}
\label{Upwards absoluteness for coinitial=omega}
Suppose $\cW_0$ is a continuous distance monoid with $\coinit{\cW_0} = \w$. Then $(\cM_0, f_0)$ has a fixed point if and only if $(\cM_1, f_1)$ has a fixed point. 
\end{theorem}
\begin{proof}
By \cref{Nice initial sequences exist} (a) there is an initial sequence $\alpha\:\w^{op} \to \nzero{W_0}$ such that $(\forall n \in \w)\, 3 \cdot \alpha(n+1) \leq \alpha(n)$. Note $\alpha$ is also an initial sequence in $\cW_1$. Let $\alpha^*\:\nzero{W_0} \to \w$ be any map such that $(\forall a \in \nzero{W_0})\, a \geq \alpha(\alpha^*(a))$. We know such a map exists as $\alpha$ is an initial sequence. 

Let $(\cN, g)$ be a generalized dynamical system let $D \subseteq N$ be $D$ dense in $\cN$ with $g[D] \subseteq D$. Let $T_{D, g\rest[D]}$ be the collection sequences $\bx = \<x_i\>_{i \in [n]} \subseteq D$ for some $n \in \w$ such that
\begin{itemize} 
\item $(\forall i,j \in [n])\, \dist[\cN](x_i, x_j) \leq \alpha(i) + \alpha(j)$. 

\item $(\forall i \in [n])\, \dist[\cN](x_i, g(x_i)) \leq \alpha(i) \text{ and } \dist[\cN](g(x_i), x_i) \leq \alpha(i)$. 

\end{itemize}
Let $\bx \sqsubseteq \by$ if $\bx$ is an initial segment of $\bx$. 

\begin{claim}
\label{Upwards absoluteness for coinitial=omega: Claim 1}
The following are equivalent. 
\begin{itemize}
\item[(a)] $(T_{D, g\rest[D]}), \sqsubseteq)$ is ill-founded. 

\item[(b)] There is an $x \in \cN$ such that $g(x) = x$. 
\end{itemize}
\end{claim}
\begin{proof}
Suppose (a) holds and $\<x_i\>_{i \in \w}$ is a sequence such that for all $n \in \w$, $\<x_i\>_{i \in [n]} \in T_{D, g\rest[D]}$. Let $A = \alpha``[\w]$ and let $y\:A \to N$ be the maps where $y(a) = x_{\alpha^{-1}(a)}$.

Now suppose $a, b \in \dom(y)$ with $a = \alpha(i)$ and $b = \alpha(j)$. We then have $y(a) = x_i$ and $y(b) = x_j$. Therefore $\dist[\cM](y(a), y(b)) = \dist[\cM](x_i, x_j) \leq \alpha(i) +\alpha(j) = a+ b$. Therefore $y$ is a Cauchy sequence. 

Suppose $y$ converges to $y^* \in N$. For $n \in \w$ there is a $\beta(n) \in \w$ such that $\beta(n) \leq n$ and $\dist[\cN](y^*, y(\alpha(\beta(n)))) \leq \alpha(n)$. Let $\gamma(n) = \alpha(\beta(n))$. We then have for all $n \in \w$,  
\begin{align*}
\dist[\cN](y^*, g(y^*)) & \leq \dist[\cN](y^*, y(\gamma(n+1))) \\
&+ \dist[\cN](y(\gamma(n+1)), g(y(\gamma(n+1)))) + \dist[\cN](g(y(\gamma(n+1))), g(y^*)) \\
&\leq \alpha(n+1) + \gamma(n+1) + \alpha(n+1) \leq 3 \cdot \alpha(n+1) \leq \alpha(n). 
\end{align*}
A similar argument shows that for all $n \in \w$, $\dist[\cN](g(y^*), y^*) \leq \alpha(n)$. Therefore $y^*$ is a fixed point of $g$ and (b) holds. 
 
Suppose (b) holds and $z \in \cN$ is such that $g(z) = z$. If $z \in D$ then for all $n \in \w$, $\<z\>_{i \in [n]} \in T_{D, g\rest[D]}$ and so $T_{D, g\rest[D]}$ is ill-founded. 

If $z \not \in D$ then, as $D$ is dense, there is a Cauchy sequence $z^+$ with range in $D$ where $z^+$ which converges to $z$. By \cref{Properties of equivalence of Cauchy sequences} (e) we can assume without loss of generality that $z^+$ is maximal. Let $y = z^+ \rest[A]$. For $i \in \w$ let $x_i = y(\alpha(i+1))$. We then have 
\begin{align*}
(\forall i, j \in \w)\, \dist[\cN](x_i, x_j) & = \dist[\cN](y(\alpha(i+1)), y(\alpha(j+1))) \\
&\leq \alpha(i+1) + \alpha(j+1) \leq \alpha(i) + \alpha(j).
\end{align*}
Further 
\begin{align*}
(\forall i \in \w)\, \dist[\cN](x_i, g(x_i)) &= \dist[\cN](y(\alpha(i)), g(y(\alpha(i)))) \\
&\leq \dist[\cN](y(\alpha(i)), z) + \dist[\cN](z, g(z)) + \dist[\cN](g(z), g(y(\alpha(i)))) \\
&\leq \alpha(i+1) + \alpha(i+1) \leq \alpha(i).
\end{align*}
A similar argument shows $(\forall i \in \w)\, \dist[\cM](g(x_i), x_i) \leq \alpha(i)$. 

Therefore for all $n \in \w$, $\<x_i\>_{i \in [n]} \in T_{D, g\rest[D]}$ and $T_{D, g\rest[D]}$ is ill-founded. Hence (a) holds.  
\end{proof}

Note that $M_0$ is dense in both $\cM_0$ and $\cM_1$ by \cref{Dense preserved by taking Cauchy completion}. Further, as the trees only depend on the generalized metric space restricted to the dense set we have $(T_{M_0, f_0})^{V_0} = (T_{M_0, f_0})^{V_1}$. 
Therefore by \cref{Upwards absoluteness for coinitial=omega: Claim 1} we have $(\cM_0, f_0)$ has a fixed point if and only if $(\cM_1, f_1)$ does. 
\end{proof}

\section{Non-Absoluteness}
\label{Section:Non-Absoluteness}

We now consider the case when the coinitiality of $\cW$ is uncountable. In this case the existence of fixed points need not be absolute. In particular we will give a concrete example of such a $\cW$-metirc space for any continuous distance space $\cW$ of coinitiality $\kappa > \w$.

\begin{definition}
Suppose $\kappa$ is an infinite ordinal. A \defn{$\kappa$-tree} is a partial order $\cT = (T, \leq)$ such that the following hold.
\begin{itemize}
\item There is a unique minimal element, called the \defn{root}.

\item For all $a \in T$, $\{b \in T \st b < a\}$ is well-founded with an order type an ordinal in $\kappa$. 
\end{itemize}

We let the \defn{level} of $a$, denoted $\lev{a}$, be the order type of $\{b \in T \st b < a\}$. If $i < \lev{a}$ we let $a\rest[i]$ be the unique element such that $a\rest[i] \leq a$ and $\lev{a\rest[i]} = i$. 

For $a, b \in T$ with $a \neq b$ let $\join(a, b) = \bigcup \{\lev{c} \st c \leq a \And c \leq b\}$. For $a \in T$ let $\join(a, a) = \kappa$. 

We call a tree \defn{pruned} if for all $a \in T$ and all $\beta < \kappa$ with $\lev{a} < \beta$ there is a $b \in T$ with $\lev{b} = \beta$ and $a \leq b$. 
\end{definition}

\begin{definition}
A \defn{path} through a $\kappa$-tree is a map $p\:\kappa \to T$ such that the following hold. 
\begin{itemize}
\item For all $\alpha \in \kappa$, $\lev{p(\alpha)} = \alpha$. 

\item For all $\alpha < \beta \in \kappa$, $p(\alpha) < p(\beta)$. 

\end{itemize}

We let $\paths[\cT]$ be the collection of paths through $T$. For $p, q \in \paths[\cT]$ let $\join(p, q) = \bigcup \{\join(a, b) \st a \in p, b \in q\}$. 

If $p \in \paths[\cT]$ and $q \in T$ let $\join(p, q) = \join(q, p) = \bigcup \{\join(a, q) \st a \in p\}$. 
\end{definition}

The following shows that if $\cofinal{\kappa} = \w$ then every pruned $\kappa$-tree must have a path. This is something that need not be the case if $\cofinal{\kappa} > \w$.

\begin{lemma}
\label{Trees with cf(kappa) = w have paths}
Suppose $\cT$ is a pruned $\kappa$-tree and $\cofinal{\kappa} = \w$. Then $\paths[\cT] \neq \emptyset$. 
\end{lemma}
\begin{proof}
Let $(\gamma_i)_{i \in \w}$ be a cofinal sequence in $\kappa$. Let $x_0$ be any element in $T$ with $\lev{x_0} = \gamma_0$. For $n\in \w$ with $x_n$ defined let $x_{n+1}$ be any element of $T$ with $x_{n} \leq x_{n+1}$ and $\lev{x_{n+1}} = \gamma_{n+1}$. Note such an $x_{n+1}$ exists as $T$ is pruned. Let $y\:\kappa \to T$ be the map where, if $\gamma_n \leq \zeta < \gamma_{n+1}$ then $y(\zeta) = x_{n+1}\rest[\zeta]$. It is then immediate that $y$ is a path through $\cT$. 
\end{proof}

We will now show how we turn any $\coinit{\cW}$-tree in a $\cW$-metric space. 

\begin{definition}
Suppose 
\begin{itemize}
\item $\cW$ is a continuous distance monoid and $\cT$ is a $\kappa$-tree. 

\item $\alpha$ is an initial sequence with $\dom(\alpha) = \kappa^{op}$ and for all $i \in \kappa$, $2 \cdot \alpha(i+1) \leq \alpha(i)$. 
\end{itemize}

Let $\cT_{\alpha} = (T \cup \paths[\cT], \dist[\cT_\alpha])$ where 
\[
(\forall x \neq y \in T)\, \dist[\cT_\alpha](x, y) = \alpha(\join(x, y)). 
\]
\end{definition}

Note $\dist[\cT_\alpha]$ is symmetric so $\dist[\cT_\alpha](x, y) = 0$ if and only if $x = y$. 

\begin{lemma}
\label{Trees with paths are Cauchy complete with T dense}
$\cT_\alpha$ is a Cauchy complete $\cW$-metric space with $T$ a dense subset. 
\end{lemma}
\begin{proof}
It is immediate from the definition fo $\dist[\cT_\alpha]$ that $T$ is a dense subset of $\cT_\alpha$. To see that $\cT_\alpha$ is Cauchy complete suppose $p\:\nzero{W} \to T \cup \paths[\cT]$ is a Cauchy sequence. As $\alpha$ is an initial sequence $p\rest[\range(\alpha)]$ is also a Cauchy sequence. It is then immediate that $p \equiv_{\cT_\alpha} p\rest[\range(\alpha)]$. 

For $i < j \in \kappa$ we have $\dist[\cT_\alpha](p(i+1), p(j)) \leq \alpha(i+1) + \alpha(j) \leq \alpha(i)$. Therefore for $0 < j \in \kappa$ if $p(j)  \in T$ we have $\lev{p(j)} \geq \sup\{i < j \st i+1 < j\}$. 

For $j \in \kappa$ if $p(j+2) \in T$ let $x(j) = p(j+2)\rest[j]$ and $x(j) = p(j+2)(j)$ otherwise. It is then straightforward to check that $p$ converges to $x$. In particular $\cT_\alpha$ is Cauchy complete. 
\end{proof}

\begin{lemma}
\label{Increasing map of pruned trees fixes all paths}
Suppose $\cT$ is a pruned tree. Let $f\:T \to T$ be a map such that $(\forall t \in T)\, f(t) > t$. 
\begin{itemize}
\item[(a)] $f$ is a non-expanding map from $(T, \dist[\cT_\alpha])$ to $(T, \dist[\cT_\alpha])$. 

\item[(b)] If $f^*$ is the unique non-expanding map from $\cT_\alpha$ to $\cT_\alpha$ extending $f$ then $(\forall p \in \paths[\cT])\, f(p) = p$. 
\end{itemize}
\end{lemma}
\begin{proof}
If $a, b \in T$ then $f(a) > a$ and $f(b) > b$ and so $\join(f(a), f(b)) \geq \join(a, b)$. Therefore $\dist[\cT_\alpha](f(a), f(b)) \leq \dist[\cT_\alpha](a, b)$. Hence $f$ is a non-expanding map and (a) holds. 

We now show (b). First note that for any $t \in T$, $\dist[\cT_\alpha](t, f(t)) = \alpha(\lev{t})$. Suppose $p \in \paths[\cT]$. Let $p^*\: \range(\alpha) \to T$ be the map where $p^*(\alpha(i)) = p(i)$ for all $i \in \kappa$. Then $p^*$ is a Cauchy sequence and $p^*$ converges to $p$. But for $i \in \kappa$, 
\begin{align*}
\dist[\cT_\alpha](p, f(p)) & \leq \dist[\cT_\alpha](p, p(i+1)) + \dist[\cT_\alpha](p(i+1), f(p(i+1))) + \dist[\cT_\alpha](f(p(i+1)), f(p)) \\
& \leq \alpha(i+1) + \alpha(i+1) + \alpha(i+1) \leq \alpha(i).
\end{align*}
Therefore $f(p) = p$ and (b) holds as $p$ was arbitrary. 
\end{proof}

In particular if $f\:\cT_\alpha \to \cT_\alpha$ is non-expanding and is such that $f(t) > t$ for all $t \in T$ then $p \in \cT_\alpha$ is a fixed point of $f$ if and only if $p \in \paths[\cT]$. Hence if we can find a tree $\cT$ which has no paths in a model of set theory $V_0$ but has paths in a model of set theory $V_1$ containing $V_0$ then the existence of a fixed point for $(\cT_\alpha,f)$ is not absolute. 

\begin{lemma}
\label{Exists pruned tree with no paths if cf(kappa) > w}
Suppose $\kappa$ is an an ordinal with uncountable cofinality. Then there is a tree $\cS_\kappa = (S_\kappa, \leq)$ which is pruned and which has no paths. 
\end{lemma}
\begin{proof}
Let $S_\kappa$ be the collection of pair $(a, g)$ such that either $a = 0$ and $g = \{(0, 0)\}$ or the following hold. 
\begin{itemize}
\item For some $n \in \w$, $g\:[n+2] \to \kappa$ is an increasing sequence of ordinals with $g(0) = 0$.  

\item $g(n) < a \leq g(n+1)$. 

\end{itemize}

Define $(a_0,g_0) \leq (a_1, g_1)$ if $g_0$ agrees with $g_1$ on $\dom(g_0)$ and $a_0 \leq a_1$. It is then straightforward to check that $\cS_\kappa =(S_\kappa, \leq)$ is a tree and $\lev{a, g} = a$. In particular $\cS_\kappa$ is a pruned $\kappa$-tree. 

Now suppose $p \in \paths[\cS_\kappa]$ to get a contradiction. Let $q = \{(n, k) \st (\exists \gamma \in \kappa)\, p(\gamma) = (a, g)\text{ and }g(n) = k\}$. It is then straightforward to check that $q\:\w \to \kappa$ is a cofinal sequence, contradicting our assumption on $\kappa$. 
\end{proof}

Putting everything together we have the following. 

\begin{theorem}
Suppose the following hold. 
\begin{itemize}
\item $V_0 \subseteq V_1$ are models of set theory. 

\item $\cW \in V_0$ is a continuous distance monoid with $(\coinit{\cW})^{V_0} > \w$. 

\item $\cofinal{\coinit{\cW}^{V_0}} = \w$.
\end{itemize}
Then there is a $\cW$-dynamical system in $V_0$ which has no fixed point but whose relativization to $V_1$ for being a generalized dynamical system does. 
\end{theorem}
\begin{proof}
By \cref{Nice initial sequences exist} there is a nice $\coinit{\cW}^{V_0}$-initial sequence for $\cW$ in $V_0$. Let $\alpha$ be one such sequence. By \cref{Exists pruned tree with no paths if cf(kappa) > w} there is a $\coinit{\cW}^{V_0}$-tree in $\cT \in V_0$ which is pruned and has no paths. By \cref{Trees with paths are Cauchy complete with T dense} we have $(\cT_\alpha, \dist[\cT_\alpha])$ is Cauchy complete in $V_0$. Let $f_0\:T \to T$ be any map such that $(\forall t \in T)\, t < f_0(t)$. Note such an $f$ exists because $\cT$ is pruned. Further note $(\cT, f_0)$ is a $\cW$-dynamical system in $V_0$ which has no fixed points. 

But by \cref{Trees with cf(kappa) = w have paths} there is a path $p \in \cT_\alpha$ in $V_1$. Further if $(\cT_\alpha^1, f_1, \mcomp{\cW})$ is the relativization of $(\cT_\alpha, f_0, \cW)$ to $V_1$ with respect to being a generalized dynamical system then by \cref{Increasing map of pruned trees fixes all paths} we have $f_1(p) = p$. 
\end{proof}

\section{Future Work}
We end with a couple of directions for future work. 
\begin{itemize}
\item[(Q1)] For which, not necessarily linear, quantales $\cW$ is having a fixed point absolute for all $\cW$-dynamical systems? (where here we are replacing distance monoids with quantales in the obvious way).

\item[(Q2)] For a fixed $\cW$, is there a characterization of those $\cW$-dynamical systems for which having a fixed point is not absolute?
\end{itemize}


\bibliographystyle{amsnomr}
\bibliography{bibliography}

\end{document}